\documentclass[11pt,a4paper]{article}%
\usepackage{geometry}
\usepackage{graphicx}
\usepackage{amsmath, amsfonts, amssymb}%
\providecommand{\U}[1]{\protect\rule{.1in}{.1in}}
\newtheorem{theorem}{Theorem}

\newtheorem{corollary}[theorem]{Corollary}

\newtheorem{lemma}[theorem]{Lemma}

\newtheorem{hypothesis}{Hypothesis}

\newtheorem{remark}[theorem]{Remark}

\newenvironment{proof}[1][Proof]{\textbf{#1.} }{\ \rule{0.5em}{0.5em}}
\newcommand{\tmop}[1]{\ensuremath{\operatorname{#1}}}
\begin{document}

\title{Full well-posedness of point vortex dynamics\\corresponding to stochastic 2D Euler equations}
\author{F. Flandoli$^{1}$, M. Gubinelli$^{2}$ , E. Priola$^{3}$\\{\small {(1) Dipartimento di Matematica Applicata ``U. Dini'', Universit\`a di
Pisa, Italia}}\\{\small { (2) CEREMADE \& CNRS (UMR 7534), Universit\'e Paris Dauphine,
France}}\\{\small { (3) Dipartimento di Matematica, Universit\`a di Torino, Italia }}}
\maketitle

\begin{abstract}
The motion of a finite number of point vortices
on a two-dimensional periodic domain is considered.
 In the deterministic case
  it is known to be well posed only for almost every
 initial configuration. Coalescence of vortices may occur for certain
initial conditions. We prove that when a \emph{generic} stochastic
perturbation compatible with the Eulerian description is introduced,
the point
vortex motion becomes  well posed for
every initial configuration, in particular coalescence disappears.

\bigskip
\noindent\textbf{MSC (2000):} 76B47 (Vortex flows) ;  60H10  (Stochastic ordinary differential equations)
\end{abstract}

\section{Introduction}

Existence and uniqueness questions for the 2D Euler equations
\begin{equation}
\frac{\partial u}{\partial t}+u\cdot\nabla u+\nabla p=0,\quad\mathrm{div}\,%
u=0,\quad u|_{t=0}=u_{0} \label{Eulerian for velocity}%
\end{equation}
are well understood in suitable functions spaces (see, for instance,
\cite{MajdaBert} and \cite{PL Lions} for a review of several results). One of
the classical results is the existence of solutions when $u_{0}$ is in the
Sobolev space $W^{1,2}$ and the uniqueness when the (scalar) vorticity
$\xi=\nabla^{\perp}\cdot u=\partial_{2}u_{1}-\partial_{1}u_{2}$ is bounded.

The case when the vorticity is a signed measure received also a lot
of attention, due to the interest in the evolution of vortex
structures like sheets or points of vorticity concentration. See
\cite{MajdaBert} for a review. Deep existence and stability results
for distributional   vorticities which do  not change sign have
been proved, first for a class of distributions which includes
vortex sheets but not vortex points, then also for point vortices
 (see among others \cite{Delhort}, \cite{Poup}). Uniqueness is an open
problem in all such cases. When the vorticity has variable sign and
is, for instance, pointwise distributed, even a reasonable
formulation of the Euler equations is missing. However, in the case
of point vortices, there are good reasons to replace the Eulerian
formulation with a Lagrangian one, based on the autonomous motion of
a finite number of point vortices.

The Lagrangian formulation of point vortex motion gives rise to a
finite dimensional ordinary differential equation, which is well
posed only for almost all initial configurations with respect to
Lebesgue measure. One can give explicit examples of initial
configurations such that different vortices coalesce in finite time.
In such a case the Lagrangian equations loose meaning. Perhaps a
proper Eulerian description could be meaningful also after the
coalescence time, but a rigorous formulation of this fact is not
known.

The purpose of this paper is to show that the previous pathology, namely the
existence of initial configurations which coalesce in finite time, is
prevented by the presence of suitable noise in the system. The point vortex
motion is well defined for all times and all initial configurations, under
suitable noise perturbations (which may be arbitrarily small). Let us describe
our aim in more detail.

As shortly recalled in the next section, Euler equations can be recast in
terms of vorticity as the system
\begin{equation}
\frac{\partial\xi}{\partial
t}+u\cdot\nabla\xi=0,\quad\xi|_{t=0}=\xi_{0},
\label{vorticity eq}%
\end{equation}%
\begin{equation}
u=-\nabla^{\perp}\Delta^{-1}\xi. \label{Biot Sav}%
\end{equation}
Concepts of weak solutions of the Euler equations are meaningful
even for distributional vorticity $\xi$, when $u$ is sufficiently
regular; square integrable is sufficient, see the theory of vortex
sheet solutions, where there are at least some existence theorems
(see \cite{MajdaBert}). The limit case when $\xi$ is the sum of
finite number of delta Dirac masses
\begin{equation}
\xi( .,t) =\sum_{i=1}^{n}\omega_{i}\delta_{x_{t}^{i}}
 \label{vorticity field}%
\end{equation}
is unfortunately too singular:\ the velocity field is not square
integrable (so even the weakest form of (\ref{Eulerian for
velocity}) is not meaningful) and its singularity coincides with the
delta Dirac points of the vorticity (so also (\ref{vorticity eq}) is
not meaningful). In spite of this, there are good arguments, based
on the limit of regular solutions supported around the ideal point
vortices \cite{March Pulv}, to accept that a certain finite
dimensional differential equation for the position of the point
vortices is the correct physical description of the evolution of
 $\xi $ in  (\ref{vorticity field}). The equations for the evolution
of the positions of point vortices have the
form%
\begin{equation}
\frac{dx_{t}^{i}}{dt}=\sum_{j\neq i}\omega_{j}K( x_{t}^{i}-x_{t}^{j}) ,\quad
i=1,...,n. \label{det point vort intro}%
\end{equation}
A few more details are explained in the next section  (see also \cite{Newton}).

If we call $X_{0}=( x_{0}^{1},...,x_{0}^{n}) $ the initial condition in
$\mathbb{R}^{2n}$ of the system of $n$ point vortices, a result of existence
and uniqueness for \emph{Lebesgue almost every} $X_{0}$ is known, see
\cite{March Pulv}. But there are examples of initial condition $X_{0}$ such
that vortices collide in finite time and a global solution does not exist.

The purpose of this research is to investigate the effect of a multiplicative
noise on the Euler equations, in the form
\begin{equation}
d\xi+u\cdot\nabla\xi dt+\sum_{k=1}^{N}\sigma_{k}( x) \cdot\nabla\xi\circ
d\beta_{t}^{k}=0,\quad\xi|_{t=0}=\xi_{0}, \label{stoch vorticity eq}%
\end{equation}
 where  $\sigma_{k}( x) $ are suitable  2d vector fields
and $\{ \beta^{k}_t\} _{k=1,...,N}$ are independent Brownian
motions. Note that   $u$ is again reconstructed from $\xi$ by
means of Biot-Savart law (\ref{Biot Sav}). Linear transport
equations are regularized by multiplicative noise, see
\cite{FlaGubPriola}: non-uniqueness phenomena of the deterministic
case disappear under the random perturbation. Our aim,  in
principle, is to prove a similar regularizing effect  for the
nonlinear problem (\ref{stoch vorticity eq}). However this is a very
difficult problem and at the moment we are not able to solve it. The
result we present in this work is in some sense a first  step and
concerns  the stochastic point vortex dynamics which corresponds to
equation (\ref{stoch vorticity eq}) and has the
form%
\begin{equation}
dx_{t}^{i}=\sum_{j\neq i}\omega_{j}K( x_{t}^{i}-x_{t}^{j}) dt+\sum_{k=1}%
^{N}\sigma_{k}( x_{t}^{i})\circ d\beta_{t}^{k},\quad i=1,...,n.
\label{stoch point eq introduction}%
\end{equation}
We prove that, under suitable assumptions on the fields
$\sigma_{k}( x) $ (those of Section \ref{section assumptions}), this
stochastic point vortex dynamics is globally well posed (in
particular coalescence of point vortices disappear) for \emph{all}
initial conditions. This is a stochastic improvement of the
deterministic theory, as in \cite{FlaGubPriola}.

For the improvements in well posedness of the linear transport equations
considered in \cite{FlaGubPriola}, it was sufficient to take constant fields
$\sigma_{k}$. Here, to avoid point vortex coalescence, we need
space-dependent  fields with a high degree of hypoellipticity, a technically
complex condition (see Hypothesis~\ref{r})  which however is generically
satisfied  (see Section~\ref{sec:generic}). Notice that in the trivial case when $\omega
_{j}=0$ for every $j=1,...,n$, system (\ref{stoch point eq introduction}) is
the so called $n$-point motion associated to the vector fields $\sigma_{k}$. A
priori this is highly degenerate; for this reason we need suitable
hypoellipticity conditions.

It would be trivial to improve the regularity of the deterministic system
(\ref{det point vort intro}) by adding independent Brownian motions to each
component, but this would not correspond to a Lagrangian point vortex
formulation of stochastic Euler equations.

Let us mention that Kotelenez~\cite[Chapter 8]{Kotelenez} considered a
similar stochastic perturbation of Euler equation and the associated point
vortex dynamics with the aim to understand the physical interest and properties of
the model. However he is not concerned with the regularizing properties of such kinds of noises.

\medskip
\noindent\textbf{Plan of the paper.} In Section~\ref{sec:two} we
explain the formal relation between  the stochastic  Euler
equation and the SDE for the point vortices;  this relation fixes
the form of the noise  allowed in the SDE. In Section~\ref{sec:main}
we state our hypothesis and prove the main result about
well-posedness of point vortex dynamics for all initial conditions.
In Section~\ref{sec:generic} we give a self-contained proof of the
fact that our hypothesis is  generically satisfied. Last we gather
in the appendix a series of well-known result on the density of the
law of SDEs under various conditions on the vector fields.

\medskip
\noindent\textbf{Acknowledgement.}  The authors would like to thank
Olivier Glass for an interesting discussion about genericity and
the transversality theorem.

\section{Stochastic 2D Euler equations and vortex dynamics}
\label{sec:two}

The aim of this section is to provide an heuristic motivation for system
(\ref{stoch point eq introduction}). A rigorous link between it and the
original stochastic Euler equations is not given in this work. It is already a
difficult problem in the deterministic case, where one of the best available
justifications is the result which states that unique solutions of Euler
equations corresponding to smoothing of the distributional initial vorticity,
converge in the weak sense of measures to (\ref{vorticity field}), see
\cite{March Pulv}. The result holds as far as point vortices do not coalesce.

Due to the difficulty of this subject, we do not aim here to give rigorous
results on the link between (\ref{stoch point eq introduction}) and
(\ref{stoch vorticity eq}), but only to provide an heuristic motivation. For
this reason, the rest of this section is not always written in rigorous terms
and we intentionally miss important details like functions spaces, regularity
of functions, etc.

\subsection{Deterministic case}

In dimension 2, the vorticity field $\xi=\nabla^{\perp}\cdot u=\partial
_{2}u_{1}-\partial_{1}u_{2}$ satisfies equation (\ref{vorticity eq}) where
$\xi_{0}=\nabla^{\perp}\cdot u_{0}$. If $\varphi$ (called potential) solves
the equation
\[
\Delta\varphi=-\xi
\]
then $u=\nabla^{\perp}\varphi$ satisfies $\xi=\nabla^{\perp}\cdot u$. Hence,
formally speaking, $u$ can be reconstructed from $\xi$ by the so called
Biot-Savart law (\ref{Biot Sav}).

Depending on the fact that we consider the equations in full space with square
integrable conditions at infinity, or on a torus with periodic boundary
conditions or other cases, one can make rigorous and uniquely defined the
previous procedure of reconstruction of $u$ from $\xi$. Let us work on the
2D-torus $\mathbb{T}= \mathbb{R}^{2}/( 2\pi\mathbb{Z}^{2}) $. Denote by $G$
the Green's function of $-\Delta$ on $\mathbb{T}$, then
$G\left(  x\right)  =\sum_{k\in\mathbb{Z}^{2}\backslash\{0\}}\left\Vert k\right\Vert
^{-2}e^{ik\cdot x}$.
The distribution $G$ is in fact a function, with a logarithmic divergence at
$x=0$, smooth everywhere else and satisfies (cf.
page 18 of \cite{March Pulv}):
\begin{align*}
C_{1}\log\vert x\vert-C_{3}  &  \leq G( x) \leq C_{2}\log\vert x\vert
+C_{3}\quad\\
\vert DG( x) \vert &  \leq C_{3}\vert x\vert^{-1},\quad\vert D^{2}G( x)
\vert\leq C_{3}\vert x\vert^{-2}%
\end{align*}
for all $x\in[ -\pi,\pi] ^{2}$, for some positive constants $C_{1},C_{2}%
,C_{3}$.

Given a periodic field $\xi$ with suitable regularity, a periodic
(distributional) solution of $\Delta\varphi=-\xi$ is given by
$
\varphi\left(  x\right)  = \int_{\mathbb{T}}  G\left(  x-y\right)  \xi\left(
y\right)  dy
$.
All other solutions differ by constants. The vector field $u=\nabla^{\perp
}\varphi$ is thus uniquely defined from $\xi$:%
\begin{equation}
u\left(  x\right)  =\int_{\mathbb{T}} K\left(  x-y\right)  \xi\left(
y\right)  dy \label{integral Biot Sav}%
\end{equation}
where
\[
K( x) =\nabla^{\perp}G( x) = \sum_{k= (k_{1},k_{2})\in\mathbb{Z}^{2}\backslash\{0\}}\frac{i
k^{\perp}}{\left\Vert k\right\Vert ^{2}}e^{ik\cdot x},\;\;\;  k^{\perp}=
(k_{2}, -k_{1}).
\]
Equation (\ref{integral Biot Sav}) is the integral form of Biot Savart law,
used throughout the paper. Since $u$ is the unique field such that $\xi
=\nabla^{\perp}\cdot u$, it is the velocity field associated to $\xi$.

A limit case of vorticity field is the distributional one given by
(\ref{vorticity field}) where $x_{t}^{i}$ is the position of point vortex $i$
at time $t$ and $\omega_{i}$ is its intensity (independent of time because
vorticity is just transported). This distributional field does not satisfy
Euler equations in the usual distributional sense: the nonlinear term
($\varphi$ is a smooth test function)%
\[
\int\xi(x,t)u(x,t)\cdot\nabla\varphi(x)dx
\]
is not well defined a priori, because the velocity field $u(x,t)$ associated
to (\ref{vorticity field}) is singular exactly at the delta Dirac points of
$\xi$. However, there are limit arguments, see \cite{March Pulv} which
rigorously motivate the following closed set of equations for the positions of
point vortices:
\begin{equation}
\frac{dx_{t}^{i}}{dt}=\sum_{j\neq i}\omega_{j}K(x_{t}^{i}-x_{t}^{j}),\quad
i=1,...,n\label{ODE}%
\end{equation}
Let us briefly explain this equation. Formally, point vortices are transported
by the fluid, hence they should satisfy $\frac{dx_{t}^{i}}{dt}=u(x_{t}^{i},t)$
where $u(x,t)$ is the velocity field associated to the vorticity field
(\ref{vorticity field}). If we put (\ref{vorticity field}) in
(\ref{integral Biot Sav}) we get
\[
u(x,t)=\sum_{j=1}^{n}\omega_{j}K(x-x_{t}^{j}).
\]
However, this expression is correct in all points $x$ different from the
vortex points themselves; notice that $K(x)$ diverges at $x=0$ so $u(x,t)$
should be properly interpreted at $x=x_{t}^{i}$. It turns out (see \cite{March
Pulv}) that the correct interpretation is
\begin{equation}
u(x_{t}^{i},t)=\sum_{j\neq i}\omega_{j}K(x_{t}^{i}-x_{t}^{j}%
).\label{velocity of point vortices}%
\end{equation}
Giving us eq.~(\ref{ODE}).

\subsection{Lack of full well-posedness}

Equations (\ref{ODE}) are not trivial since the vector field is
divergent when two particles collide; and there is no repulsion (but
also no attraction) when particles approach each other.
Nevertheless, in the periodic case, equations (\ref{ODE}) are
well-posed for almost every initial condition $X_{0}=(
x_{0}^{1},...,x_{0}^{n}) $ with respect to Lebesgue measure on the
product space. One can interpret this result by saying that the
system is almost surely well-posed when initial conditions are
chosen at random in a uniform way.

In the whole space the same result is known under the additional assumption that $\{
\omega_{i}\} _{i=1,...,n}$ satisfy $\sum_{i\in\pi}\omega_{i}\neq0$ for all
$\pi\subset\{ 1,...,n\} $.

The restriction to almost every initial condition is not just a weakness of
the technique:  there are counterexamples in the form of explicit initial configurations which
collide in finite time, see~\cite{March Pulv}.

\subsection{Stochastic case}

We shall prove well-posedness for every initial condition under a
suitable random perturbation. The idea behind is simply that the
noise makes the same effect of a randomization of the initial
conditions. However, there is a subtle difference between these two
randomizations. In the case of initial conditions, Lebesgue measure
in product space is used, thus the initial positions of particles
are perturbed independently one of each other. In the noise case, it
would be trivial to perturb each particle independently: this
produces immediately a strong regularization which ultimately would
imply well-posedness. Such kind of noise, however, has no meaning in
terms of Euler equation. What we want is a noisy version of Euler
equation which is solvable in the case of distributional point
vortex fields. When we write the noise at the level of the point
vortex dynamics, \emph{the noise is the same for all vortices} (but
since it is a space dependent noise, it is computed ad different
spatial points). This is in principle  a  source of difficulties.

\vskip 1mm We consider the stochastic equation with multiplicative
noise (\ref{stoch vorticity eq}) where $\sigma_{k}( x) $ are 2D
smooth vector fields and $\{ \beta^{k}_t\} _{k=1,...,N}$  are a
finite sequence  of independent Brownian motions defined  on a
stochastic basis $( \Omega,F,( F_{t}) ,P) $ (fixed once and for
all).

The associated dynamics
of $n$ point vortices is the stochastic system in $\mathbb{R}^{2n}$
\begin{equation}
dx_{t}^{i}=\sum_{j\neq i}\omega_{j}K( x_{t}^{i}-x_{t}^{j}) dt+\sum_{k=1}%
^{N}\sigma_{k}( x_{t}^{i}) \circ d\beta_{t}^{k},\quad x_{0}^{i}=x^{i} \label{SDE}%
\end{equation}
for $i=1,...,n$, with each single $x_{t}^{i}$ in $\mathbb{R}^{2}$ and where $\circ$ denote Stratonovich integration.

Let us formally show that the measure valued vorticity field
(\ref{vorticity field}), with $x_{t}^{i}$ given by a solution of equations
(\ref{SDE}), solve (\ref{stoch vorticity eq}). The weak form of the stochastic
Euler equation is (assuming $\mathrm{div}\,\sigma_{k}=0$, $k =1, \ldots, N$)
\[
d\langle\xi,\varphi\rangle=\langle\xi,u\cdot\nabla\varphi\rangle dt+\sum
_{k=1}^{N}\langle\xi,\sigma_{k}\cdot\nabla\varphi\rangle\circ d\beta_{t}^{k}%
\]
where we have denoted by $\varphi$ a smooth test function and by
$\langle.,.\rangle$ the dual pairing. Let $\xi$ be the distribution defined by
(\ref{vorticity field}). From It\^{o} formula for $\varphi( x_{t}^{i}) $ in
Stratonovich form we get
\begin{align*}
d\langle\xi,\varphi\rangle &  =\sum_{i=1}^{n}\omega_{i}d\varphi( x_{t}^{i})
=\sum_{i=1}^{n}\omega_{i}\nabla\varphi( x_{t}^{i}) \sum_{j\neq i}\omega_{j}K(
x_{t}^{i}-x_{t}^{j}) dt\\
&  +\sum_{i=1}^{n}\omega_{i}\sum_{k=1}^{N}\nabla\varphi( x_{t}^{i}) \sigma
_{k}( x_{t}^{i}) \circ d\beta_{t}^{k}.
\end{align*}
The second term is the right one:
\[
\sum_{k=1}^{N}\langle\xi,\sigma_{k}\cdot\nabla\varphi\rangle\circ d\beta
_{t}^{k}=\sum_{i=1}^{n}\omega_{i}\sum_{k=1}^{N}\sigma_{k}( x_{t}^{i})
\cdot\nabla\varphi( x_{t}^{i}) \circ d\beta_{t}^{k}.
\]
As to the first term, recall that the velocity field $u$ associated to point
vortices is given by (\ref{velocity of point vortices}). Then
\begin{align*}
\langle\xi,u\cdot\nabla\varphi\rangle &  =\sum_{i=1}^{n}\omega_{i}u( x_{t}%
^{i}) \cdot\nabla\varphi( x_{t}^{i})
 =\sum_{i=1}^{n}\sum_{j\neq i}\omega_{i}\omega_{j}K( x_{t}^{i}-x_{t}^{j})
\cdot\nabla\varphi( x_{t}^{i})
\end{align*}
and thus also the first term is the right one. This completes the heuristic proof.

A very important remark,  already mentioned in the previous section,  is that
the noise in this system is the same for all particles. Thus this is very
similar to the so called $n$-point motion of a single SDE. The regularizing
effect of the noise at the level of the $n$-point motion is a very non-trivial
fact. For instance, the easiest non-degenerate noise, namely the simple
additive one ($\sigma_{k}$ are 2D vectors)
\[
dx_{t}^{i}=\sum_{j\neq i}\omega_{j}K( x_{t}^{i}-x_{t}^{j}) dt+\sum_{k=1}%
^{N}\sigma_{k}d\beta_{t}^{k}%
\]
cannot yield any better result than the deterministic case, because the change
of variables  $y_{t}^{i}=x_{t}^{i}-\sum_{k=1}^{N}\sigma_{k}d\beta_{t}^{k}$
leads to the equation
\[
dy_{t}^{i}=\sum_{j\neq i}\omega_{j}K( y_{t}^{i}-y_{t}^{j}) dt
\]
which is exactly the deterministic one. If collapse happens for an initial
condition of this equation, the same initial condition produces collapse in
the previous SDE.

On the contrary, a strongly space dependent noise may contrast collapse. When
point vortices come close one to the other, the noise should be sufficiently
un-correlated (at small distances)\ to perturb in a generic way the motion of
the two vortices and produce the same effect of a random perturbation of
initial conditions.

\section{Main results and proofs}
\label{sec:main}
\subsection{Regularization by noise}\label{section assumptions}

We consider system (\ref{SDE}) on the 2D-torus
$\mathbb{T}=\mathbb{R}^{2}/( 2\pi\mathbb{Z}^{2}) $. It is a
$C^{\infty}$  compact connected Riemannian manifold with the smooth
metric induced by Euclidean metric of $\mathbb{R}^{2}$. In fact, for
simplicity, we may assume we work on the full space $\mathbb{R}^{2}$
and all the vector fields and functions are $2\pi $-periodic, but
sometimes the interpretation as a compact manifold is more
illuminating. We will consider a fixed choice $\{ \omega_{j}\}
_{j=1,...,n}\subset$ $\mathbb{R}$ of vortex intensities.

Let $\Gamma$ be the set of all $\left(  x^{1},...,x^{n}\right)  \in
{\mathbb{T}}^{n}$ such that $x^{i}=x^{j}$ for some $i\neq j$  ($\Gamma$  is
the union of the generalized diagonals of $\mathbb{T}^{n}$).
Let $\left\{  \sigma_{k}\right\}  _{k=1,...,N}$ be a finite number of smooth
vector fields on ${\mathbb{T}}$. Introduce the associated vector fields on
$\mathbb{T}^{n}$:
\begin{equation}
A_{\sigma_{k}}(x^{1},\ldots,x^{n})=A_{k}(x^{1},\ldots,x^{n})=\big(\sigma
_{k}(x^{1}),\ldots,\sigma_{k}(x^{n})\big)\label{dopo}%
\end{equation}
Recall that given vector fields $A,B$ in $\mathbb{R}^{m}$, their Lie bracket $[A,B]$ is the vector fields in
$\mathbb{R}^{m}$ defined by
\[
[ A,B] =( A\cdot\nabla) B-( B\cdot\nabla) A.
\]

We assume that $\left\{  \sigma_{k}\right\}  _{k=1,...,N}$ satisfies:

\begin{hypothesis}
\label{r}
\parbox{0cm}{}
\begin{enumerate}
\item The vector fields $\sigma_{k}$ are periodic, infinitely differentiable and $\mathrm{div}%
\sigma_{k}=0$ ;

\item (Bracket  generating condition) The vector space spanned by the vector fields
\[
A_{1},...,A_{N},\qquad[ A_{i},A_{j}] ,1\leq i,j\leq N,\qquad[ A_{i},[
A_{j},A_{k}] ] ,1\leq i,j,k\leq N,...
\]
at every point $x\in\Gamma^{c}$ is $\mathbb{R}^{2n}$.
\end{enumerate}
\end{hypothesis}

The second assumption is a form of H\"{o}rmander's condition. It
will ensures that the law at  any time $t>0$
of the solution of a
regularized stochastic  equation is absolutely continuous with respect to
Lebesgue measure if we start outside the diagonal $\Gamma$ (see also
the appendix).

Under this hypothesis we are able to prove the following result of
well-posedness of the dynamics for  all initial $n$-point
configurations.

\begin{theorem}
\label{main theorem}Under Hypothesis~\ref{r},
%the assumptions of section \ref{section assumptions}%
for all
$X_{0}=(x_{0}^{1},...,x_{0}^{n})\in{\mathbb{T}}^{n}\backslash\Gamma$
equation (\ref{SDE}) has one and only one global strong solution.
\end{theorem}

Before going to the proofs,
let us make some remarks on our
hypothesis. The bracket generation condition appears already in  the
papers \cite{BaxStroock} and \cite{DolgKK} which study the
asymptotic behavior of stochastic flows (among
other properties, the
Lyapunov exponents and the large deviations for additive
functionals). In \cite[Section 2]{DolgKK} it is stated that the
bracket generating condition in Hypothesis~\ref{r} is \emph{generic}
among smooth vector fields. We have been unable to find a proof of
this statement in the literature (both stochastic or more
dynamical-system oriented) and so in Section~\ref{sec:generic} we give a
self-contained and elementary argument which justifies  this
statement.

\begin{remark}
Instead of the noise taken from \cite{BaxStroock},\cite{DolgKK}, it is
natural to consider an infinite dimensional noise $W( x,t) =\sum_{k=1}%
^{\infty}\sigma_{k}( x) \beta_{t}^{k}$, for instance the isotropic
divergence free Brownian field which generates the isotropic
Brownian motion, see~\cite{BaxHarris},\cite{LeJan},\cite{LeJan Ray}.
Here we restrict our attention to finite dimensional noise for which
we already know results about absolute continuity of fixed time
marginals.  The interesting fact about infinite dimensional
noise is that it is easy to constructs explicit noises which are
``full'' outside $\Gamma$ and for which it is reasonable to expect
the validity of density results on the law.
\end{remark}

The proof of Theorem~\ref{main theorem} goes through  the study of a
regularized problem where the singular Biot-Savart kernel is
replaced by a smooth one. We first prove well-posedness for almost
every initial conditions (as in the deterministic setting) and then,
exploiting the existence   of a density for the law at fixed time,
we improve to well-posedness for all initial conditions.

\subsection{Regularization}

For sufficiently small $\delta$, let $G^{\delta}( x) $ be a smooth $2\pi
$-periodic function (hence bounded with its derivatives) such that, on $[
-\pi,\pi] ^{2}$,
\[
G^{\delta}( x) =G( x) \text{ for }\vert x\vert>\delta.
\]
Set $K^{\delta}=\nabla^{\perp}G^{\delta}$. We shall use the following
quantitative properties, beside smoothness:
\begin{align*}
C_{1}\log( \vert x\vert\vee\delta) -C_{3}  &  \leq G^{\delta}( x) \leq
C_{2}\log( \vert x\vert\vee\delta) +C_{3}\quad\\
\vert DG^{\delta}( x) \vert &  \leq C_{3}( \vert x\vert\vee\delta) ^{-1}%
,\quad\vert D^{2}G^{\delta}( x) \vert\leq C_{3}( \vert x\vert\vee\delta) ^{-2}%
\end{align*}
for all $x\in[ -\pi,\pi] ^{2}$, for some positive constants $C_{1},C_{2}%
,C_{3}$ (possibly different from those of the same inequalities for $G$ but independent of $\delta$).
We consider the regularized equation
\begin{equation}
dx_{t}^{i,\delta}=\sum_{j\neq i}\omega_{j}K^{\delta}( x_{t}^{i,\delta}%
-x_{t}^{j,\delta}) dt+\sum_{k=1}^{N}\sigma_{k}( x_{t}^{i,\delta})\circ d\beta
_{t}^{k}. \label{delta SDE}%
\end{equation}
which in It\^o form reads
\begin{equation}
dx_{t}^{i,\delta}=\sum_{j\neq i}\omega_{j}K^{\delta}( x_{t}^{i,\delta}%
-x_{t}^{j,\delta}) dt+\frac12 \sum_{k=1}^{N}(\sigma_{k}\cdot \nabla\sigma_{k})( x_{t}^{i,\delta}) dt +\sum_{k=1}^{N}\sigma_{k}( x_{t}^{i,\delta}) d\beta
_{t}^{k}. \label{delta SDE-Ito}%
\end{equation}

We immediately have: for every $X_{0}=( x_{0}^{1},...,x_{0}^{n}) \in
\mathbb{R}^{2n}$, there exists a unique strong solution $( X_{t}^{X_{0}})
_{t\geq0}$ to\ this equation in $\mathbb{R}^{2n}$. We even have a smooth
stochastic flow $\varphi_{t}^{\delta}$ on ${\mathbb{T}}^{n}$, see
\cite{Kunita}.

\subsection{Measure conservation}

Denote the divergence in $\mathbb{R}^{2}$ by $\mathrm{div}_{2}$\ and in
$\mathbb{R}^{2n}$ by $\mathrm{div}_{2n}$.We have
\[
\mathrm{div}_{2n}[ \sum_{j\neq i}\omega^{i}K^{\delta}( x_{i}-x_{j}) ]
_{i=1,...,n}=\sum_{i=1}^{n}\mathrm{div}_{2}[ \sum_{j\neq i}\omega^{i}%
K^{\delta}( x_{i}-x_{j}) ] =0
\]
because $K^{\delta}=\nabla^{\perp}G^{\delta}$. The same is true without
regularization. Moreover,
\[
\mathrm{div}_{2n}[ \sum_{k=1}^{N}\sigma_{k}( x^{i}) \beta_{t}^{k}]
_{i=1,...,n}=\sum_{i=1}^{n}\mathrm{div}_{2}\sigma_{k}( x^{i}) \beta_{t}^{k}=0
\]
because $\mathrm{div}_{2}\sigma_{k}=0$. By classical computations on the
smooth flow $\varphi_{t}^{\delta}$, one can check that its Jacobian
determinant is equal to one, as a consequence of the previous divergence free
conditions. Hence we have:

\begin{lemma}
\label{conservation}For every integrable function $h$ on ${\mathbb{T}}^{n}$,
we have
\[
\int_{{\mathbb{T}}^{n}}h( \varphi_{t}^{\delta}( X_{0}) ) dX_{0}=\int
_{{\mathbb{T}}^{n}}h( Y) dY.
\]

\end{lemma}

\subsection{Estimates about coalescence}

Denote by $[ x,y] _{t}$ the mutual quadratic covariation of two continuous
semimartingales $( x_{t}) _{t\geq0}$ and $( y_{t}) _{t\geq0}$. Denote by
$x^{\alpha}$, $\alpha=1,2$, the two components of an element of ${\mathbb{T}}$
in the coordinate frame coming from Euclidean coordinates. If $(
x_{t}^{i,\delta}) $ is a solution of equation (\ref{delta SDE}), then
\begin{align*}
&  [ ( x^{i,\delta}-x^{j,\delta}) ^{\alpha},( x^{i,\delta}-x^{j,\delta})
^{\beta}] _{t}
=\sum_{k=1}^{N}\int_{0}^{t}( \sigma_{k}^{\alpha}( x_{s}^{i,\delta})
-\sigma_{k}^{\alpha}( x_{s}^{j,\delta}) ) ( \sigma_{k}^{\beta}( x_{s}%
^{i,\delta}) -\sigma_{k}^{\beta}( x_{s}^{j,\delta}) ) ds.
\end{align*}

\begin{lemma}
Let $\varphi_{t}^{\delta}$, be the flow on ${\mathbb{T}}^{n}$ associated to
(\ref{delta SDE}). Let $g^{\delta}:{\mathbb{T}}^{n}\rightarrow\mathbb{R}$ be
the function
\[
g^{\delta}( X) =-\sum_{i,j=1...,n,\;i\neq j}G^{\delta}( x^{i}-x^{j}) ,\quad
X=( x^{1},...,x^{n}) .
\]
Then there exists a non negative integrable function $h$ on ${\mathbb{T}}^{n}$
such that
\[
E[ \sup_{t\in[ 0,T] }g^{\delta}( \varphi_{t}^{\delta}( X_{0}) ) ] \leq
g^{\delta}( X_{0}) +\int_{0}^{T}E[ h( \varphi_{t}^{\delta}( X_{0}) ) ] dt.
\]

\end{lemma}

\begin{proof}
From It\^{o} formula we have
\[
g^{\delta}( \varphi_{t}^{\delta}( X_{0}) ) =g^{\delta}( X_{0}) -\sum
_{i,j=1...,n,\;i\neq j}( I_{1a}^{ij}( t) +I_{1b}^{ij}( t) +I_{2}^{ij}( t)
+I_{3}^{ij}( t)+I_{4}^{ij}( t) )
\]
where
\[
I_{1a}^{ij}( t) =\int_{0}^{t}\sum_{i^{\prime}\neq i}\omega_{i^{\prime}%
}K^{\delta}( x_{s}^{i,\delta}-x_{s}^{i^{\prime},\delta}) \cdot\nabla
G^{\delta}( x_{s}^{i,\delta}-x_{s}^{j,\delta}) ds
\]%
\[
I_{1b}^{ij}( t) =  - \int_{0}^{t}\sum_{j^{\prime}\neq j}\omega_{j^{\prime}%
}K^{\delta}( x_{s}^{j,\delta}-x_{s}^{j^{\prime},\delta}) \cdot\nabla
G^{\delta}( x_{s}^{i,\delta}-x_{s}^{j,\delta}) ds
\]%
\[
I_{2}^{ij}( t) =\sum_{k=1}^{N}\int_{0}^{t}( \sigma_{k}( x_{s}^{i,\delta})
-\sigma_{k}( x_{s}^{j,\delta}) ) \cdot\nabla G^{\delta}( x_{s}^{i,\delta
}-x_{s}^{j,\delta}) d\beta_{s}^{k}%
\]%
\[
I_{3}^{ij}( t) =\frac{1}{2}\sum_{\alpha,\beta=1}^{2}\int_{0}^{t}\frac
{\partial^{2}G^{\delta}}{\partial x^{\alpha}\partial x^{\beta}}(
x_{s}^{i,\delta}-x_{s}^{j,\delta}) d[ ( x^{i,\delta}-x^{j,\delta}) ^{\alpha},(
x^{i,\delta}-x^{j,\delta}) ^{\beta}] _{s}
\]
\[
I_{4}^{ij}( t) =\frac12\int_{0}^{t}\nabla G^{\delta}(
x_{s}^{i,\delta}-x_{s}^{j,\delta})\cdot [ (\sigma_k\cdot \nabla\sigma_k)(x_s^{i,\delta})-(\sigma_k\cdot\nabla \sigma_k)(x_s^{j,\delta})] d s.
\]
Since $K^{\delta}=\nabla^{\perp}G^{\delta}$ and $\nabla^{\perp}G^{\delta}( x)
$ is orthogonal to $\nabla G^{\delta}( x) $, we have
\[
I_{1a}^{ij}( t) =\int_{0}^{t}\sum_{i^{\prime}\neq i,\;i^{\prime}\neq j}%
\omega_{i^{\prime}}K^{\delta}( x_{s}^{i,\delta}-x_{s}^{i^{\prime},\delta})
\cdot\nabla G^{\delta}( x_{s}^{i,\delta}-x_{s}^{j,\delta}) ds
\]%
\[
I_{1b}^{ij}( t) = - \int_{0}^{t}\sum_{j^{\prime}\neq j,\;j^{\prime}\neq
i}\omega_{j^{\prime}}K^{\delta}( x_{s}^{j,\delta}-x_{s}^{j^{\prime},\delta})
\cdot\nabla G^{\delta}( x_{s}^{i,\delta}-x_{s}^{j,\delta}) ds.
\]
Hence
\[
| I_{1a}^{ij}( t) | \leq C\int_{0}^{t}\sum_{i^{\prime}\neq i,\;i^{\prime}\neq
j}( | x_{s}^{i,\delta}-x_{s}^{i^{\prime},\delta}| \vee\delta) ^{-1}( |
x_{s}^{i,\delta}-x_{s}^{j,\delta}| \vee\delta) ^{-1}ds
\]%
\[
| I_{1b}^{ij}( t) | \leq C\int_{0}^{t}\sum_{j^{\prime}\neq j,\;j^{\prime}\neq
i}( | x_{s}^{j,\delta}-x_{s}^{j^{\prime},\delta}| \vee\delta) ^{-1}( |
x_{s}^{i,\delta}-x_{s}^{j,\delta}| \vee\delta) ^{-1}ds
\]%
\[
\sum_{i,j=1...,n,\;i\neq j}( | I_{1a}^{ij}( t) | +| I_{1b}^{ij}( t) | ) \leq
C\int_{0}^{t}h_{1}^{\delta}( \varphi_{s}^{\delta}( X_{0}) ) ds
\]
where
\[
h_{1}^{\delta}( X) =\sum_{\substack{i,j,l=1...,n\\i\neq j,\;l\neq i,\;l\neq
j}}( | x^{i}-x^{l}| \vee\delta) ^{-1}( | x^{i}-x^{j}| \vee\delta) ^{-1}%
\]
with $X=( x^{1},...,x^{n}) $. Setting
\[
h_{1}( X) =\sum_{\substack{i,j,l=1...,n\\i\neq j,\;l\neq i,\;l\neq j}}( |
x^{i}-x^{l}| ) ^{-1}( | x^{i}-x^{j}| ) ^{-1}%
\]
we have that $h_{1}$ is integrable over ${\mathbb{T}}^{n}$, and $h_{1}%
^{\delta}( X) \leq h_{1}( X) $ for all $X\in{\mathbb{T}}^{n}$. Moreover, By
BDG inequality and the smoothness of $\sigma_{k}$ we have
\begin{align*}
E\left[  \sup_{t\in\left[  0,T\right]  }\left|  I_{2}^{ij}\left(  t\right)
\right|  \right]   &  \leq CE\left[  \left[  I_{2}^{ij},I_{2}^{ij}\right]
^{1/2} _{T}\right] \\
&  \leq C E \Big [ \Big(\int_{0}^{T}\left(  \left|  x_{s}^{i,\delta}%
-x_{s}^{j,\delta}\right|  \vee\delta\right)  ^{-2}\left|  x_{s}^{i,\delta
}-x_{s}^{j,\delta}\right|  ^{2}ds \Big)^{1/2} \Big]
\end{align*}
hence
$
\sum_{i,j=1...,n,\;i\neq j}E[ \sup_{t\in[ 0,T] }| I_{2}^{ij}( t) | ] \leq C
$.
Finally,
\[
 \sup_{t\in [  0,T ]  }  |  I_{3}^{ij}(  t)  |  \leq C\int_{0}^{T}(  |
x_{s}^{i,\delta}-x_{s}^{j,\delta}|  \vee\delta)  ^{-2}|
x_{s}^{i,\delta}-x_{s}^{j,\delta}|  ^{2}ds
\]
and
\[
 \sup_{t\in [  0,T ]  } |  I_{4}^{ij}(  t)  |  \leq C\int_{0}^{T}(  |
x_{s}^{i,\delta}-x_{s}^{j,\delta}|  \vee\delta)  ^{-1}|
x_{s}^{i,\delta}-x_{s}^{j,\delta}| ds
\]
so again
$
\sum_{i,j=1...,n,\;i\neq j}(| I_{3}^{ij}( t) |+| I_{4}^{ij}( t) |) \leq C
$.
Summarizing,
\[
E\left[  \sup_{t\in\left[  0,T\right]  }g^{\delta}\left(  \varphi_{t}^{\delta
}\left(  X_{0}\right)  \right)  \right]  \leq g^{\delta}\left(  X_{0}\right)
\]%
\[
+\sum_{i,j=1...,n,\;i\neq j}E [  \sup_{t\in [  0,T ]  } (
 |  I_{1a}^{ij} (  t )   |  +  |  I_{1b}^{ij} (
t )   |  + |  I_{2}^{ij} (  t )   |  + |
I_{3}^{ij} (  t )   |  + |
I_{4}^{ij} (  t )   |   )   ]
\]%
\[
\leq g^{\delta}\left(  X_{0}\right)  + C\int_{0}^{T}E\left[  h_{1}\left(
\varphi_{s}^{\delta}\left(  X_{0}\right)  \right)  \right]  ds+C.
\]
The proof is complete.
\end{proof}

\begin{corollary}
There is a constant $C>0$ such that for every $\delta>0$%
\[
E\int_{{\mathbb{T}}^{n}}\sup_{t\in[ 0,T] }[ g^{\delta}( \varphi_{t}^{\delta}(
X_{0}) ) ] dX_{0}\leq C<\infty.
\]

\end{corollary}

\begin{proof}
From the previous lemma and lemma \ref{conservation} we have
\begin{align*}
E\int_{{\mathbb{T}}^{n}}\sup_{t\in[ 0,T] }[ g^{\delta}( \varphi_{t}^{\delta}(
X_{0}) ) ] dX_{0}%
& \leq\int_{{\mathbb{T}}^{n}}g^{\delta}( X_{0}) dX_{0}+E\int_{0}^{T}[
\int_{{\mathbb{T}}^{n}}h( \varphi_{t}^{\delta}( X_{0}) ) dX_{0}] dt\\
&  =\int_{{\mathbb{T}}^{n}}g^{\delta}( X_{0}) dX_{0}+T\int_{{\mathbb{T}}^{n}%
}h( Y) dY.
\end{align*}
The proof is complete.
\end{proof}

\begin{corollary}
There exists a constant $C>0$ such that for all $\varepsilon,\delta>0$ we
have
\[
( \lambda_{{\mathbb{T}}^{n}}\otimes P) ( \inf_{i\neq j}\inf_{t\in[ 0,T] }|
x_{t}^{i,\delta}-x_{t}^{j,\delta}| \leq\varepsilon) \leq-\frac{C}{\log(
\varepsilon\vee\delta) }.
\]

\end{corollary}

\begin{proof}
We have%
\begin{align*}
g^{\delta}( \varphi_{t}^{\delta}( X_{0}) )  &  =-\sum_{i,j=1...,n,\;i\neq
j}G^{\delta}( x_{t}^{i,\delta}-x_{t}^{j,\delta})\\
&  \geq-\sum_{i,j=1...,n,\;i\neq j}C_{2}\log( \vert x_{t}^{i,\delta}%
-x_{t}^{j,\delta}\vert\vee\delta) -\frac{n( n-1) }{2}C_{3}.
\end{align*}
Given $\varepsilon,\delta>0$, smaller than one, if
$
\inf_{i\neq j}\inf_{t\in[ 0,T] }\vert x_{t}^{i,\delta}-x_{t}^{j,\delta}%
\vert\leq\varepsilon
$
namely if there are $t_{0}\in[ 0,T] $ and $i_{0}\neq j_{0}$ such that $\vert
x_{t_{0}}^{i_{0},\delta}-x_{t_{0}}^{j_{0},\delta}\vert\leq\varepsilon$ then%
\[
g^{\delta}( \varphi_{t_{0}}^{\delta}( X_{0}) ) \geq-C_{2}\log( \varepsilon
\vee\delta) -\frac{n( n-1) }{2}( C_{2}\log2\pi\sqrt{2}+C_{3})
\]
(we have used the fact that $\log( \vert x_{t}^{i,\delta}-x_{t}^{j,\delta
}\vert\vee\delta) \leq\log2\pi\sqrt{2}$) and thus
\[
\sup_{t\in[ 0,T] }g^{\delta}( \varphi_{t}^{\delta}( X_{0}) ) \geq-C_{2}\log(
\varepsilon\vee\delta) -C_{4}n^{2}.
\]
By Chebyshev inequality (notice that $-C_{2}\log( \varepsilon\vee\delta) >0$)
and the previous lemma,
\begin{align*}
&  ( \lambda_{{\mathbb{T}}^{n}}\otimes P) ( \inf_{i\neq j}\inf_{t\in[ 0,T]
}\vert x_{t}^{i,\delta}-x_{t}^{j,\delta}\vert\leq\varepsilon)\\
&  \leq( \lambda_{{\mathbb{T}}^{n}}\otimes P) ( \sup_{t\in[ 0,T] }g^{\delta}(
\varphi_{t}^{\delta}( X_{0}) ) +C_{4}n^{2}\geq-C_{2}\log( \varepsilon
\vee\delta) )\\
&  \leq-\frac{C_{5}n^{2}}{\log( \varepsilon\vee\delta) }.
\end{align*}
The proof is complete.
\end{proof}

\begin{remark}
The function $h$ and the constants $C$ of the previous statements depend on
the number $n$ of point vortices and the time interval $[ 0,T] $.
\end{remark}

\subsection{Well-posedness for Lebesgue almost every initial condition}

As a first consequence of the previous estimates, we prove the same result of
the deterministic case.

Recall that $\Gamma$ is the singular set in ${\mathbb{T}}^{n}$ for the vortex
dynamics, namely the set of all $( x^{1},...,x^{n}) \in{\mathbb{T}}^{n}$ such
that $x^{i}=x^{j}$ for some $i\neq j$. The drift of equation (\ref{SDE}) is
well defined only on $\Gamma^{c}$. Thus the notion of strong solution $(
X_{t}) _{t\geq0}$ to equation (\ref{SDE}) is the classical one for SDEs with
the addition of the condition that%
\[
P( X_{t}\in\Gamma^{c}\text{ for all }t\geq0) =1.
\]

\begin{theorem}
\label{teo 1}For Lebesgue almost every $X_{0}=( x_{0}^{1},...,x_{0}^{n})
\in{\mathbb{T}}^{n}$, equation (\ref{SDE}) has one and only one global strong solution.
\end{theorem}

\begin{proof}
Denote by $\Gamma_{\delta}$ the closed $\delta$-neighbor of $\Gamma$ in
${\mathbb{T}}^{n}$. Given $X_{0}\in\Gamma_{\delta}^{c}$, denote by
$\tau_{X_{0}}^{\delta}( \omega) $ the first instant when $\varphi_{t}^{\delta
}( X_{0}) \in\Gamma_{\delta}$ and set it equal to $+\infty$ if this fact never
happens. We have $P\left(  \tau_{X_{0}}^{\delta}>0\right)  =1$ by continuity
of trajectories. The solution $\varphi_{t}^{\delta}\left(  X_{0}\right)  $, on
the random interval $\left[  0,\tau_{X_{0}}^{\delta}\right]  $, is also the
unique solution $\left(  X_{t}\right)  $ of equation (\ref{SDE}).  Thus
$\tau_{X_{0}}^{\delta}( \omega) $ is also the first instant when $X_{t}%
\in\Gamma_{\delta}$. Set%
\[
\tau_{X_{0}}( \omega) =\sup_{\delta\in( 0,1) }\tau_{X_{0}}^{\delta}( \omega)
.
\]
By localization, we have a unique solution of equation (\ref{SDE}) on
$[0,\tau_{X_{0}})$. If we prove that $P( \tau_{X_{0}}=\infty) =1$ for a.e.
$X_{0}$, we have proved the theorem. Given $T>0$ and $\delta^{\ast}>0$, it is
sufficient to prove that for Lebesgue a.e. $X_{0}\in\Gamma_{\delta^{\ast}}%
^{c}$, we have $P( \tau_{X_{0}}\geq T) =1$.

Form the last corollary of the previous section we know that%
\[
( \lambda_{{\mathbb{T}}^{n}}\otimes P) ( \inf_{i\neq j}\inf_{t\in[ 0,T] }|
x_{t}^{i,\delta}-x_{t}^{j,\delta}| \leq\delta) \leq-\frac{C}{\log\delta}.
\]
Let $\{ \delta_{k}\} _{k\in\mathbb{N}}$ be a sequence such that the series
$\sum_{k=1}^{\infty}\frac{1}{\log\delta_{k}}$ converges. Take it such that
$\delta_{k}\leq\delta^{\ast}$ for all $k\in\mathbb{N}$. By Borel-Cantelli
lemma, there is a measurable set $N\subset{\mathbb{T}}^{n}\times\Omega$ with
$( \lambda_{{\mathbb{T}}^{n}}\otimes P) ( N) =0$, such that for all $(
X_{0},\omega) \in N^{c}$ there is $k_{0}=k_{0}( X_{0},\omega) \in\mathbb{N}$
such that for all $k\geq k_{0}( X_{0},\omega) $%
\[
\inf_{i\neq j}\inf_{t\in[ 0,T] }| \varphi_{t}^{i,\delta_{k}}( X_{0}) ( \omega)
-\varphi_{t}^{j,\delta_{k}}( X_{0}) ( \omega) | >\delta_{k}%
\]
where $\varphi_{t}^{i,\delta_{k}}( X_{0}) $ is $x_{t}^{i,\delta_{k}}$ when the
initial condition is $X_{0}$. If we restrict ourselves to $( X_{0},\omega) \in
N^{c}\cap( \Gamma_{\delta^{\ast}}^{c}\times\Omega) $, the previous statement
implies $\tau_{X_{0}}^{\delta_{k}}( \omega) \geq T$ for all $k\geq k_{0}(
X_{0},\omega) $. This implies
\[
\tau_{X_{0}}( \omega) \geq T.
\]
We have proved this inequality for all $( X_{0},\omega) \in N^{c}\cap(
\Gamma_{\delta^{\ast}}^{c}\times\Omega) $, namely for almost every $(
X_{0},\omega) $ in $\Gamma_{\delta^{\ast}}^{c}\times\Omega$. By Fubini-Tonelli
theorem, there is a measurable set $\Delta\subset\Gamma_{\delta^{\ast}}^{c}$
with $\lambda_{{\mathbb{T}}^{n}}( \Delta) =1$, such that for all $X_{0}%
\in\Delta$ we have $\tau_{X_{0}}( \omega) \geq T$ with $P$-probability one.
The proof is complete.
\end{proof}

\subsection{Improvement due to the noise}

We may now prove our main result,
Theorem~\ref{main theorem}.

\begin{proof}  (Theorem~\ref{main theorem})
Given $X_{0}\in\Gamma^{c}$, a strong unique local solution on $[0,\tau_{X_{0}%
})$ exists ($\tau_{X_{0}}$ defined in the proof of theorem \ref{teo 1}). Let
us add a point $\Delta$ to ${\mathbb{T}}^{n}$ and set $\varphi_{t}( X_{0})
=\Delta$ for $t\geq\tau_{X_{0}}$, where $\tau_{X_{0}}<\infty$. The family of
processes $\varphi_{t}( X_{0}) $, $X_{0}\in\Gamma^{c}$, so defined, lives in
$\Gamma^{c}\cup\Delta$ for positive times and is Markov.  Then
\[
P( \varphi_{[ \varepsilon,T] }( X_{0}) \in\Gamma^{c}) \newline=\int
_{\Gamma^{c}\cup\{ \Delta\} }P( \varphi_{[ 0,T-\varepsilon] }( Y) \in
\Gamma^{c}) \mu_{\varphi_{\varepsilon}( X_{0}) }( dY)
\]
where $\{  \varphi_{[ \varepsilon,T] }( X_{0}) \in\Gamma^{c} \}  = \{
\omega\in\Omega\; :\; \varphi_{t} (X_{0})(\omega) $ $\in\Gamma^{c},$ for any
$\; t \in[ \varepsilon,T] \}$ and $\mu_{\varphi_{\varepsilon}( X_{0}) }$ is
the law of $\varphi_{\varepsilon}( X_{0}) $. Denote by $N\subset{\mathbb{T}%
}^{n}$ a measurable set such that all initial conditions in $N^{c}$ give rise
to a well posed Cauchy problem. We have
\[
P( \varphi_{[ 0,T-\varepsilon] }( Y) \in\Gamma^{c}) =1
\]
for all $Y\in N^{c}$. Then%
\begin{align*}
&  P( \varphi_{[ \varepsilon,T] }( X_{0}) \in\Gamma^{c}) \geq\int_{N^{c}}P(
\varphi_{[ 0,T-\varepsilon] }( Y) \in\Gamma^{c}) \mu_{\varphi_{\varepsilon}(
X_{0}) }( dY)\\
&  =1-\mu_{\varphi_{\varepsilon}( X_{0}) }( N) .
\end{align*}

Now, assume $X_{0}\in\Gamma_{\delta^{\ast}}^{c}$ for some $\delta^{\ast}>0$.
We have, for all $\delta\in(0,\delta^{\ast})$,%
\begin{align*}
\mu_{\varphi_{\varepsilon}(X_{0})}(N) &  =P(\varphi_{\varepsilon}(X_{0})\in
N)\\
&  =P(\varphi_{\varepsilon}(X_{0})\in N,\tau_{X_{0}}^{\delta}>\varepsilon
)+P(\varphi_{\varepsilon}(X_{0})\in N,\tau_{X_{0}}^{\delta}\leq\varepsilon)\\
&  \leq P(\varphi_{\varepsilon}^{\delta}(X_{0})\in N,\tau_{X_{0}}^{\delta
}>\varepsilon)+P(\tau_{X_{0}}^{\delta}\leq\varepsilon)\\
&  \leq P(\varphi_{\varepsilon}^{\delta}(X_{0})\in N)+P(\tau_{X_{0}}^{\delta
}\leq\varepsilon)\\
&  =P(\tau_{X_{0}}^{\delta}\leq\varepsilon).
\end{align*}
To say that $P(\varphi_{\varepsilon}^{\delta}(X_{0})\in N)=0$ we have used two
facts: $N$ is Lebesgue-negligible, the law of $\varphi_{t}^{\delta}(X_{0})$ on
${\mathbb{T}}^{n}$ is absolutely continuous with respect to Lebesgue measure,
for each $X_{0}\in\Gamma^{c}$, $\delta>0$, $t>0$. The latter property is a
consequence of the second main assumption of section \ref{section assumptions}%
. See the appendix~\ref{sec:remarks} for details; we apply, in particular,
Theorem~\ref{Ito}.

Just by continuity of trajectories, we have $\lim_{\varepsilon\rightarrow
0}P(\tau_{X_{0}}^{\delta}\leq\varepsilon)=0$. Hence%
\[
\lim_{\varepsilon\rightarrow0}P(\varphi_{\lbrack\varepsilon,T]}(X_{0}%
)\in\Gamma^{c})=1.
\]
The family of events $(\varphi_{\lbrack\frac{1}{n},T]}(X_{0})\in\Gamma^{c})$
is decreasing in $n$, hence $P(\varphi_{\lbrack\frac{1}{n},T]}(X_{0})\in
\Gamma^{c})$ is also decreasing. This implies $P(\varphi_{\lbrack
\varepsilon,T]}(X_{0})\in\Gamma^{c})=1$ for every $\varepsilon$ giving
$P(\varphi_{\lbrack0,T]}(X_{0})\in\Gamma^{c})=1$.
\end{proof}

\subsection{Variations on the result  }

Let us complete this section with a variant of the previous result.
Next section is devoted to the proof that the assumptions of Section
\ref{section assumptions} are generic. But in fact we prove more,
namely that generically it happens that the vector fields
$A_{1},...,A_{N}$ themselves span $\mathbb{R}^{2n}$ at every point
$x\in\Gamma^{c}$ is $\mathbb{R}^{2n}$ (no Lie brackets are needed).
It is thus meaningful to investigate the problem under the following
assumption: $\left\{  \sigma_{k}\right\}  _{k=1,...,N}$ satisfies:\

\begin{hypothesis}
\label{r 2}
\parbox{0pt}{}
\begin{enumerate}
\item $\sigma_{k}$ are periodic, $C^2$ and
$\mathrm{div}\sigma_{k}=0$

\item the vector space spanned by the vector fields $A_{1},...,A_{N}$ at every
point $x\in\Gamma^{c}$ is $\mathbb{R}^{2n}$.
\end{enumerate}
\end{hypothesis}

Item 2 of this assumption is more restrictive than the corresponding
one of Hypothesis~\ref{r}, but smoothness of the fields is no more
needed. Under
Hypothesis~\ref{r 2}, we still have that the law of $\varphi_{t}^{\delta}%
(X_{0})$ on ${\mathbb{T}}^{n}$ (see the notations of the previous sections) is
absolutely continuous with respect to Lebesgue measure, for each $X_{0}%
\in\Gamma^{c}$, $\delta>0$, $t>0$; we use now Corollary~\ref{Corollary 2 Appendix}. 
Let us also remark that the proof of absolute
continuity of the law under this assumption is more elementary than under
Hypothesis~\ref{r}. For all these reasons it is worth to state also the
following variant of Theorem \ref{main theorem} (the proof is the same, based
on the previous remark on the absolute continuity).

\begin{theorem}
Under Hypothesis~\ref{r 2}, for all $X_{0}=(x_{0}^{1},...,x_{0}^{n}%
)\in{\mathbb{T}}^{n}\backslash\Gamma$ equation (\ref{SDE}) has one and only
one global strong solution.
\end{theorem}

\section{Generic $n$-point motions are hypoelliptic}
\label{sec:generic}

In this section we are going to provide a self-contained proof of the following statement which stipulates that $n$-point motions satisfying our assumptions are \emph{generic}.

\begin{theorem}
\label{th:npoint-genericity}
For all $M > 2 n$ there exists a residual set $Q \subset(C^{\infty})^{2 n M}$
such that  for every $(f_{a, i})_{a = 1, \ldots, M, i = 1, \ldots, 2 n} \in Q$
we have  $\operatorname{span}  \{A_{f_{a, i}} (x)\}_{a = 1, \ldots, M, i = 1,
\ldots, 2 n}  =\mathbb{R}^{2 n}$ for every $x \in\Gamma^{c}$.
\end{theorem}

The parametric Sard's theorem (or Thom's transversality theorem) are  general
tools which allow to prove generic properties of geometric objects (see, for
instance~\cite{GuPol}); here we intend properties valid for almost all objects
with respect to some natural measure, or valid in a residual set (countable
intersection of open dense sets). For some applications of transversality to control theory the reader could look at~\cite{KSD} where some interesting examples are worked out in a quite explicit setting.

Here we consider an easy version of the theorem which
we are going to use to show that for a sufficiently large but otherwise
generic family of vector fields on the torus, the associated $n$-point motion
generate, as a Lie algebra, the full tangent space in each point outside the diagonals.
The basic idea is simple,  unfortunately we haven't found a reference to an equivalent statement which do not require some background in differential topology to be understood, so
we provide here the easy proof for reader sake.

\begin{theorem}
\label{th:transv}  Let $\ell< n$ and let $X \subset\mathbb{R}^{\ell}$ and $Y
\subset\mathbb{R}^{m}$ be open sets.  Consider a $C^{1}$ function $F : X
\times Y \rightarrow\mathbb{R}^{n}$  and assume that $0$ is a \emph{regular
value} for $F$ (i.e. the Jacobian matrix $D F  (x, y)$ is surjective for all
$(x, y) \in F^{- 1} (0)$). Then the set  $\mathcal{X}_{y} =  \{x \in X : F (x,
y) = 0\}$ is empty for Lebesgue almost every  $y \in Y$.
\end{theorem}

\begin{proof}
 Consider a point $(x_0, y_0) \in F^{- 1} (0)$. By the implicit function theorem
 and the fact that $\tmop{dim} \tmop{Im} DF (x_0, y_0) = n$ there exists open sets
$U, V$ such that  $x_0 \in U \subset X$ and $y_0 \in V \subset Y$ and for which
the set $(U \times V) \cap F^{- 1} (0)$ is the graph of a $C^1$ function defined on the open set
$W \subset \mathbb{R}^{\ell + m - n}$. In particular there exists an differentiable
homeomorphism $\psi : W \rightarrow
 (U \times V) \cap F^{- 1} (0) \subset \mathbb{R}^{\ell} \times
 \mathbb{R}^m$. Let $\pi_2 : X \times Y \rightarrow
 Y$ be the canonical projection over the second factor and consider the differentiable map $\pi_1 \circ \psi : W \rightarrow
 \mathbb{R}^m$: the image $W' = \pi_1 \circ \psi (W) \subset V$ of $W$ is a set of dimension
 $m + \ell - n < m$ and then of zero measure with respect to the $m$-dimensional Lebesgue measure.
 From an covering of $F^{- 1}
 (0)$ by open sets of the form $U \times V$ we can then obtain a finite subcover and form the union of the associated $W'$s which we call
 $\tilde{Y} \subset \mathbb{R}^m$ and which is still a negligible set. Now
 $\tilde{Y}$ contains exactly the points $y \in Y$ such that there exists $x \in X$ for which
  $F (x, y) = 0$, so we conclude that $y \in Y \backslash \tilde{Y} \Rightarrow
 \mathcal{X}_y = \emptyset$.
\end{proof}

\medskip

For every $d \in\mathbb{N}$ define the finite-dimensional real vector space
$\mathcal{F}_{d}$ of the solenoidal  vector fields $f : \mathbb{T}
\rightarrow\mathbb{R}^{2}$ on the torus $\mathbb{T}$ of the form
$
f (x) = \sum_{|k_{1} |,|k_2| \leqslant d}  k^{\bot} e^{i
\langle k, x \rangle} \hat{f} (k)
$.
Fix $n \geqslant1$ and recall that $\Gamma=\{(x^{1}, \ldots, x^{n}) : \min_{i
\neq j} |x^{i} - x^{j} | = 0\} \subseteq\mathbb{T}^{n}$.  Let $D =
\dim(\mathcal{F}_{d}) = (2d + 1)^{2}$.
According to \eqref{dopo}, for every vector field $f : \mathbb{T}
\rightarrow\mathbb{R}^{2}$ on $\mathbb{T}$ define $A_{f}$ as the vector field
on $\mathbb{T}^{n}$ given by $A_{f} (x) = (f (x^{1}), \ldots, f (x^{n}))$.

To understand how to use Theorem~\ref{th:transv}  to prove
genericity results for vector fields let us give a simple result
which helps in understanding the main argument.

\begin{lemma}
Let $d \in\mathbb{N}$. Fix a point $x \in\mathbb{T}^{2 n}$ and assume that
there exist vector fields $h_{1}, \ldots, h_{2 n} \in\mathcal{F}_{d}$ such
that the family $\{A_{h_{i}}  (x)\}_{i = 1, \ldots, 2 n}$ span all
$\mathbb{R}^{2 n}$. Then the same is true for Lebesgue almost every vector
fields $\sigma_{1}, \ldots, \sigma_{2 n} \in\mathcal{F}_{d}$ (i.e., we have
that  $\{A_{\sigma_{i}} (x)\}_{i = 1, \ldots, 2 n}$ spans all $\mathbb{R}^{2}$,
for  a.e. $\sigma_{1}, \ldots, \sigma_{2 n} \in\mathcal{F}_{d}$).
\end{lemma}

\begin{proof}
Consider the map $\Psi:  \mathcal{F}_{d}^{2 n} \times
\mathbb{R}^{2n} \rightarrow\mathbb{R}^{2 n}\times \mathbb{R}$
\[
\Psi(\sigma_{1}, \ldots, \sigma_{2 n}, u) = (u_{1} A_{\sigma_{1}} (x)
+ \cdots+ u_{2 n} A_{\sigma_{2 n}} (x) , \sum_{i=1}^{2n} u_i^2 -1).
\]
If we show that $\operatorname{rank}  (D \Psi(\sigma_{1}, \ldots, \sigma_{2
n}, u)) = 2  n+1$ for every $(\sigma_{1}, \ldots, \sigma_{2 n}, u)\in \Psi^{-1}(0,0)$  then we have
that the set  of vector fields $\sigma_{1}, \ldots, \sigma_{2 n}$ such that
$\Psi(\sigma_{1}, \ldots,  \sigma_{2 n}, u) = 0$ for some $u \in
\mathbb{R}^{2n}$ such that  $|u|=1$ is of zero Lebesgue measure  which allows us to conclude that
$\{A_{\sigma_{i}} (x)\}_{i = 1, \ldots, 2 n}$ span all
$\mathbb{R}^{2 n}$ for almost every choice  $\sigma$ in
$\mathcal{F}_{d}^{2n}$. Let us then compute $D \Psi$.  Taking
$q = (q_1, q_2)
\in\mathbb{Z}^2$, $ |q_1| \le d,\; |q_2| \le d$,
we have
\[
D_{\widehat{\sigma_{}}_{i} (q)} \Psi(\sigma_{1}, \ldots, \sigma_{2 n}, u) = (
u_{i} D_{\widehat{\sigma_{}}_{i} (q)} A_{\sigma_{i}} (x),0) = (u_{i}   ( q^{\bot} e^{i \langle q, x^{1} \rangle}, \ldots,q^{\bot} e^{i \langle q, x^{n} \rangle}), 0)
\]
and
\[
D_{u_{i}} \Psi(\sigma_{1}, \ldots, \sigma_{2 n}, u) = (A_{\sigma_i}(x),2 u_i)
\]
Since $|u|^2=1$ at least one of the components $u_i \neq 0$ so that the span of these vectors contains all the elements of the form
$
(h (x^{1}), \ldots, h (x^{n})) =  (A_{h} (x),\rho)
$
for arbitrary $h \in\mathcal{F}_{d}$ and $\rho\in\mathbb{R}$.
But by assumption these vectors span $\mathbb{R}^{2n}\times\mathbb{R}$ so we can conclude
using Theorem~\ref{th:transv}.
\end{proof}

\medskip

Let us now return to our main aim: build families of vector fields spanning
$\mathbb{R}^{2n}$ in each point of $\Gamma^{c}$.
The neighborhoods of the diagonals $\Gamma$ are source of troubles so for the
moment let us restrict to the consideration of $n$-point configurations
belonging to an open set $G\subset\mathbb{T}^{2n}$ away from them.

Consider the map $\Phi: \mathcal{F}_{d}^{2 n M} \times G \times
(\mathbb{R}^{2n})^{M} \rightarrow(\mathbb{R}^{2 n}\times \mathbb{R})^{M}$ given by
\[
\Phi(F, x,U)  = ( (\sum_{i = 1}^{2 n} u_{1, i} A_{f_{1, i}} (x),|u_1|^2-1), \cdots, (\sum_{i =
1}^{2 n} u_{M, i} A_{f_{M, i}} (x),|u_M|^2-1)) .
\]
where $F = (f_{1, 1}, \ldots f_{1, 2 n}, \ldots, f_{M, 2 n}) \in\mathcal{F}%
_{d}^{2 n M}$ and $U = (u_{1}, \ldots, u_{M}) \in(
{\mathbb{R}^{2n}_0})^{M}$. Then
\[
D \Phi(F, x, U) : \mathbb{R}^{2 nMD} \times\mathbb{R}^{2 n} \times
\mathbb{R}^{2 n M} \rightarrow(\mathbb{R}^{2 n}\times\mathbb{R})^{M}
\]
The various components
of the Jacobian matrix   are given by (we denote by $\mathbb{I}_{a =
b}$
the indicator
function)
\[
(D_{u_{a, i}} \Phi(F, x, U))_{b} =\mathbb{I}_{a = b} (A_{f_{a, i}}
(x),2 u_{a,i}),
%\;\;\; i = 1, \ldots, 2n,\;\; a = 1, \ldots, M,
\]
\[
(D_{\hat{f}_{a, i} (q)} \Phi(F, x, U))_{b} =\mathbb{I}_{a = b} (u_{a, i}
D_{\hat{f}_{a, i} (q)} (f_{a, i} (x^{1}), \ldots, f_{a, i} (x^{n})),0)
\]
\[
=\mathbb{I}_{a = b} (u_{a, i} (q^{\bot} e^{i \langle q_{}, x^{1} \rangle}, \ldots,
q^{\bot} e^{i \langle q_{}, x^{n} \rangle}),0)
\]
\[
(D_{x^{i}} \Phi(F, x, U))_{a} =  \sum_{j = 1}^{2 n} (u_{a,
j}  (f_{a, j} (x^{1}), \ldots, D_{x^{i}} f_{a, j} (x^{i}), \ldots, f_{a, j}
(x^{n})),0)
\]
where $a, b = 1, \ldots, M$. The image of $D \Phi(F, x, U)$ contains then
vectors $v$ of the form
\[
v_{b} = \sum_{ a,i,q} \lambda_{a,i, q} (D_{\hat{f}_{a, i} (q)} \Phi(F, x, U))_{b}
=(\sum_{i} u_{b, i} (g_{b, i} (x^{1}), \ldots, g_{b, i} (x^{n})),0)
\]
with $a, b = 1, \ldots, M$, with arbitrary coefficients $\lambda_{a,
i, q}$ and where $g_{a, i} (x) = \sum_{q} \lambda_{a, i, q} q^{\bot}
e^{i \langle q_{}, x^{} \rangle} $ are  arbitrary vectors in
$\mathcal{F}^{d}$. Now note that for any $a = 1, \ldots, M$ the constraint $|u_{a}|^2=1$
imply that there exists $i = 1, \ldots, 2 n$ such that $u_{a, i}
\neq0$. This allows to conclude that in the image of $D \Phi(F, x,
U)$ belong all the vectors  $((A_{h_{1}} (x),\rho_1), \ldots,
(A_{h_{M}} (x),\rho_M))$ for an arbitrary family $\{h_{a}
\in\mathcal{F}_{d}\}_{a = 1, \ldots, M}$ and $\rho_1,..,\rho_M\in\mathbb{R}$. Now we use the assumption
that for any $x \in G$ we have vector fields $\sigma_{1}, \ldots,
\sigma_{2 n}$ such that $\{A_{\sigma_{i}} (x)\}_{i = 1, \ldots, 2
n}$ span all $\mathbb{R}^{2 n}$. This is enough to conclude that for
every $(F, x, U)$ we have $\operatorname{Im}  (D \Phi(F, x, U)) =
(\mathbb{R}^{2 n}\times\mathbb{R})^{M}$.

\medskip Now, by using Theorem~\ref{th:transv} we deduce that for
Lebesgue-almost every $F\in \mathcal{F}_{d}^{2 n M}$ the set of
configurations $x \in G$ and auxiliary vectors $U
\in({\mathbb{R}^{2n}})^{M}$ such that $\Phi(F, x, U) = 0$ is
empty. This in turn implies that for almost every realization of
Fourier coefficients the $2 n M$ vector fields $\{A_{f_{a, i}} \}_{a
= 1, \ldots, M, i = 1, \ldots, 2 n}$ span $\mathbb{R}^{2 n}$ in each
point $x \in G$ since for every $x\in G$ at least one of the
combinations
$
\sum_{i = 1}^{2 n} u_{1, i} A_{f_{1, i}} (x), \ldots, \sum_{i = 1}^{2 n}
u_{M, i} A_{f_{M, i}} (x)
$
is always different from zero for any possible choice of $u_{a, i}$. We just
proved that

\begin{theorem}
Assume that for every $x \in G$ there exists vector-fields $\sigma_{1},
\ldots,  \sigma_{2 n} \in\mathcal{F}_{d}$ such that $\operatorname{Span}
\{A_{\sigma_{i}} (x)\}_{i = 1, \ldots, 2 n}=\mathbb{R}^{2 n}$. Then for any $M
2 n$ and for Lebesgue almost every  realization of $2 n M$ vector fields
$\{f_{a, i} \in\mathbb{F}_{d} \}_{a = 1, \ldots, M, i = 1, \ldots, 2 n}$, the
family $\{A_{f_{a, i}} \}_{a = 1, \ldots, M, i = 1, \ldots, 2 n}$ spans
$\mathbb{R}^{2 n}$ in all the points $x \in G$.
\end{theorem}

Note that this theorem allows us to obtain a result valid in every point for a
generic set of vector fields form a construction of a set of vector fields
specific for each point, which is a lot easier to do.

For every $\delta> 0$ let us now define the open set $G_{\delta} =\{x
\in\mathbb{T}^{2 n} : \min_{i \neq j} |x^{i} - x^{j} | > \delta\}
\subset\mathbb{T}^{2 n}$ of points $\delta$-uniformly away from the diagonals.

A simple construction gives that for each $\delta>0$ there exist two smooth
divergence-free vector fields $g_{1}(x)$ and $g_{2}(x)$ with compact support
inside the ball $B(0,\delta/2)$ and such that $g_{1}(0)=(1,0)$ and
$g_{2}(0)=(0,1)$ (it is sufficient to use fields of the form $g\left(
x\right)  =\varphi( \vert x-x_{0}\vert ^{2}) (
x-x_{0})  ^{\perp}$ with suitable $x_{0}\in\mathbb{R}^{2}$ and smooth
scalar compact support function $\varphi$). In such a way, for any fixed point
$\hat{x}\in G_{\delta}$ we can obtain $2n$ vector fields $f_{1},\ldots,f_{2n}$
of the form
\[
f_{2k-1}(x)=g_{1}(x-\hat{x}_{k}),f_{2k}(x)=g_{2}(x-\hat{x}_{k}),k=1,\ldots,n
\]
such that $\{A_{f_{i}}(\hat{x})\}_{i=1,\ldots,2n}$ is the canonical basis of
$\mathbb{R}^{2n}$. A difficulty stems from the fact that these fields do not
necessarily belong to $\mathcal{F}_{d}$ for some $d$. We need then to
approximate the functions $g_{1}$ and $g_{2}$ by elements of $\mathcal{F}_{d}$
for $d$ large enough. Fix $\varepsilon>0$ small enough, by density of
trigonometric polynomials, there exists $d>0$ and $\widetilde{g_{1}%
},\widetilde{g_{2}}\in\mathcal{F}_{d}$ such that $\sup_{x\in\mathbb{T}^{2}%
}|g_{i}(x)-\tilde{g}_{i}(x)|<\varepsilon$. Note that the functions
\[
\tilde{f}_{2k-1}(x)=\tilde{g}_{1}(x-\hat{x}_{k}),\tilde{f}_{2k}(x)=\tilde
{g}_{2}(x-\hat{x}_{k}),k=1,\ldots,n
\]
belong to $\mathcal{F}_{d}$ for any $\hat{x}\in\mathbb{T}^{2}$ and that, for
example,
\[
|A_{\tilde{f}_{1}(\hat{x})}-(1,0,\ldots,0)|\leqslant C\varepsilon
\]
where the constant does not depend on the parameters of the problem. Then for
$\varepsilon$ small enough, the family $\{A_{\tilde{f}_{i}(\hat{x}%
)}\}_{i=1,\ldots,2n}$ spans all $\mathbb{R}^{2n}$. The value of $d$ depends
only on $\varepsilon$ and $\delta$ but not on $\hat{x}\in G_{\delta}$. This leads us to the following result.

\begin{lemma}
For each $\delta> 0$ there exists $d \geqslant1$ such that for every $x \in
G_{\delta}$ we can find $2 n$ vector fields $f_{1}, \ldots, f_{2 n}
\in\mathcal{F}_{d}$ with the property that $\operatorname{span}  \{A_{f_{i}
(x)} \}=\mathbb{R}^{2 n}$.
\end{lemma}

An easy implication is then

\begin{corollary}
For every $\delta> 0$ and $d > d_{0} (\delta)$, almost every realization of $2
n  M$ vector fields $\{f_{a, i} \in\mathcal{F}_{d} \}_{a = 1, \ldots, M, i =
1, \ldots, 2 n}$ is such that $\operatorname{span}  \{A_{f_{a, i}} (x)\}_{a =
1, \ldots, M, i = 1, \ldots, 2 n} =\mathbb{R}^{2 n}$ for all $x \in G_{\delta
}$.
\end{corollary}

By approximation of $C^{\infty}$ vector fields by elements in $\mathcal{F}%
_{d}$ we can conclude that also the set $Q_{\delta} \subset(C^{\infty
}(\mathbb{T}^{2};\mathbb{R}^{2}))^{2 n M}$ of $2 n M$ vector fields $\{f_{a,
i} \}_{a = 1, \ldots, M, i = 1, \ldots, 2 n}$ such that for all $x \in
G_{\delta}$ $\operatorname{span}  \{A_{f_{a, i}} (x)\}_{a = 1, \ldots, M, i =
1, \ldots, 2 n} =\mathbb{R}^{2 n}$ is dense in $(C^{\infty})^{2 n M}$.

Let us prove that $Q_{\delta}$ contains an open dense subset. For any compact
$K\subset\mathbb{T}^{2}$ define $Q_{K}$ as the subset of $(C^{\infty
}(\mathbb{T}^{2};\mathbb{R}^{2}))^{2 n M}$ which spans the full tangent space in
every point of $K$.

We first prove that the sets $Q_{K}$ are open: indeed assume that there exists
a sequence $(f^{(k)}_{i,a}) \in Q_{K}^{c}$ such that $f^{(k)}$ converge to a
point $f$ in $Q_{K}$. For each $f^{(k)}$ there exists a point $x^{(k)} \in K$
for which $\operatorname{span} (A_{f^{(k)}_{i,a}}(x^{(k)})) \neq
\mathbb{R}^{2n}$. By compactness of $K$ we can extract a subsequence, still
denoted by $(x^{k})_{k\ge1}$ which converges to $x\in K$. Then by uniform
convergence of $f^{(k)}$ to $f$ we deduce that we also have
$\operatorname{span} (A_{f_{i,a}}(x)) \neq\mathbb{R}^{2n}$ which is in
contradiction with the fact that $f\in Q_{K}$.

Then observe that for any $0< \delta^{\prime}< \delta$ there exists a compact
$K$ such that $G_{\delta} \subset K \subset G_{\delta^{\prime}}$ and then that
$Q_{\delta^{\prime}} \subset Q_{K} \subset Q_{\delta}$. The set $Q_{\delta
^{\prime}}$ is dense and contained in an open set $Q_{K}$ which proves that
the interior of $Q_{\delta}$ is both open and dense, that is a residual set
(or co-meagre).

At this point, by countable intersection, we get that $Q = \cap_{k} Q_{1 / k}$
is also residual and its elements are exactly the vector fields such that
$\operatorname{span}  \{A_{f_{a, i}} (x)\} =\mathbb{R}^{2 n}$ in every point
of $\Gamma^{c}$.
We have then proved Theorem~\ref{th:npoint-genericity}.

\appendix

\section{Remarks on hypoellipticity}

\label{sec:remarks} We want to clarify the role of the nondegeneracy condition
of the $n$-point  motion assumed in section \ref{section assumptions}.
Let us recall the following theorem. See for instance \cite{Nualart}, Theorem 2.3.2.

\begin{theorem}
\label{th:hormander}
Consider the stochastic equation in Stratonovich form in $\mathbb{R}^{m}$%
\[
X_{t}=x_{0}+\sum_{j=1}^{N}\int_{0}^{t}A_{j}( X_{s}) \circ dW_{s}^{j}+\int
_{0}^{t}A_{0}( X_{s}) ds
\]
with infinitely differentiable coefficients with bounded derivatives of all
order. Assume the following H\"{o}rmander's condition at point $x_{0}$: the
vector space spanned by the vector fields
\[
A_{1},...,A_{N},\qquad[ A_{i},A_{j}] ,0\leq i,j\leq N,\qquad[ A_{i},[
A_{j},A_{k}] ] ,0\leq i,j,k\leq N,...
\]
at point $x_{0}$ is $\mathbb{R}^{m}$. Then, for every $t>0$, the law of
$X_{t}$ is absolutely continuous with respect to the Lebesgue measure.
\end{theorem}

When the vector fields $A_{1},...,A_{N}$ themselves span $\mathbb{R}^{m}$,
there is a simpler criterium, due to \cite{BouHir}. We recall a simplified
version of Theorem 2.3.1 from \cite{Nualart}. Denote by $A( x) $ the $m\times
N$ matrix with $A_{1}( x) ,...,A_{N}( x) $ as columns and by $\sigma( x) $ the
$m\times m$ matrix $A( x) A( x) ^{T}$.

\begin{theorem}
Let $(X_{t})_{t\geq0}$ be a solution of the It\^{o} equation  in
$\mathbb{R}^{m}$%
\begin{equation}
X_{t}=x_{0}+\sum_{j=1}^{N}\int_{0}^{t}A_{j}(X_{s})dW_{s}^{j}+\int_{0}^{t}%
A_{0}(X_{s})ds\label{Ito}%
\end{equation}
with
globally Lipschitz coefficients. Assume
\[
P(\int_{0}^{t}1_{\{\det\sigma(X_{s})\neq0\}}ds>0)=1
\]
for all $t>0$. Then, for every $t>0$, the law of $X_{t}$ is absolutely
continuous with respect to the Lebesgue measure.
\end{theorem}

\begin{corollary}
\label{Corollary 2 Appendix}Let $(X_{t})_{t\geq0}$ be a solution of the
It\^{o} equation (\ref{Ito}), with globally Lipschitz coefficients. If
$A_{1}(x),...,A_{N}(x)$ generate $\mathbb{R}^{m}$ at $x=x_{0}$, then, for
every $t>0$, the law of $X_{t}$ is absolutely continuous with respect to the
Lebesgue measure.
\end{corollary}

\begin{proof}
Since the fields are continuous, $A_{1}( x) ,...,A_{N}( x) $ generate
$\mathbb{R}^{m}$ at all points of a neighbor $\mathcal{U}$ of $x_{0}$. The
solution $( X_{t}) _{t\geq0}$ has continuous paths, thus belongs to
$\mathcal{U}$ at least over a small random time interval $[ 0,\tau] $, $P(
\tau>0) =1$. On $\mathcal{U}$\ we have $\det\sigma( x) \neq0$, hence the
assumption of the theorem is satisfied. The proof is complete.
\end{proof}

\end{document}